\documentclass[11pt,reqno,a4paper]{amsart}
\usepackage{euscript}
\usepackage{comment}

\theoremstyle{plain}
\numberwithin{equation}{section}
\newtheorem{theorem}{Theorem}
\newtheorem{proposition}{Proposition}

\newtheorem{definition}{Definition}

\theoremstyle{remark}
\newtheorem{remark}{Remark}
\newtheorem{example}{Example}
\renewcommand{\epsilon}{\varepsilon}
\renewcommand{\phi}{\varphi}
\DeclareMathOperator{\Ima}{Im}
\DeclareMathOperator{\Ker}{Ker}

\def\Id{\text{\rm Id}}

\usepackage{xcolor}
\usepackage{amsthm}
\usepackage{amsmath}
\usepackage{amsfonts}

\begin{document}
\title[$h$-dichotomies for evolution families]{$h$-dichotomies via  noncritical uniformity and expansiveness for evolution families}
\author[Davor Dragi\v cevi\' c]{Davor Dragi\v cevi\' c } 
\address{Department of Mathematics, University of Rijeka, Croatia}
\email[Davor Dragi\v cevi\' c]{ddragicevic@math.uniri.hr}
\maketitle

\begin{abstract}
In a recent paper (Math. Ann. 393 (2025), 1769--1795), Elorreaga et al. have obtained a complete characterization of the notion of a $h$-dichotomy for ordinary differential equations on a finite-dimensional space in terms of the notions of $h$-expansiveness and $h$-noncriticality. Their results extended the previous results of Coppel and Palmer, which dealt with exponential dichotomies. The main objective of this note is to extend the results of Elorreaga et al. to arbitrary invertible evolution families that act on Banach spaces. We emphasize that our approach is completely different and considerably simpler from the one developed by Elorreaga et al. It is based on the time-rescaling method introduced by Dragi\v cevi\' c and Silva.
\end{abstract}

\section{Introduction}
The notion of an exponential dichotomy (essentially introduced by Perron~\cite{Per}) plays a fundamental role in the qualitative theory of \emph{nonautonomous} differential equations and dynamical systems. It can be described as a nonautonomous counterpart to the classical notion of hyperbolicity as it requires that the phase space of a linear dynamics splits (at every moment of time) into two directions: the stable and the unstable direction. Along the stable direction,  dynamics exhibits exponential contraction forward in time, while along the unstable direction it exhibits the same property backward in time (which corresponds to the exponential expansiveness forward in time). We refer to the important monographs~\cite{CL,Coppel,DK,Henry,Pot,SY-book-2002} for a detailed exposition of various aspects of the theory of exponential dichotomies, including many applications.

Despite its importance, due to the flexibility of nonautonomous dynamics, it is fairly easy to construct broad classes of dynamics which exhibit behavior similar to exponential dichotomies but where the rates of contraction (resp. expansion) along stable (resp. unstable) directions are not necessarily exponential. To our knowledge, Martin Jr.~\cite{Martin}, Muldowney~\cite{M} and Naulin and Pinto~\cite{NP} were the first to systematically study such generalized dichotomies, in which the rates of contraction and expansion are prescribed by some general functions (growth rates).

In their recent paper~\cite{EPR} (for related earlier work see~\cite{WXZ}),   the authors have obtained important new characterizations of the notion of a (uniform) $h$-dichotomy for nonautonomous ordinary differential equations (which is a special case of a more general notion of the $(h,k)$-dichotomy introduced in~\cite{NP}). These characterizations are given in terms of the newly introduced notions of uniform $h$-noncriticality and $h$-expansiveness, motivated by the notions of uniform noncriticality and exponential expansiveness studied by Coppel~\cite{Coppel} and Palmer~\cite{Palmer}. In particular, the results from~\cite{EPR} extend some of those obtained by Palmer~\cite{Palmer} from exponential to general $h$-dichotomies. Their approach relied on the idea that to any $h$-dichotomy we can associate a totally ordered topological group. This idea first appeared in~\cite{PV}.

The main objective of the present paper is to obtain versions of the results of~\cite{EPR} for general invertible evolution families that act on Banach spaces. Although it is likely that one can adapt the approach from~\cite{EPR} to our setting, we prefer to use a different approach, which relies on the so-called time-rescaling. The basic idea is that the existence of an $h$-dichotomy of some evolution family $T(t,s)$ is equivalent to the existence of an exponential dichotomy of a new evolution family $T_h(t,s)$. This technique was developed in~\cite{DS} (building on the earlier work~\cite{DSS} for polynomial dichotomies) for the case of discrete time, while in the present paper we developed it for continuous time. We stress that some aspects of this approach were hidden in the arguments of previous works (see, for example,~\cite{DLP}) but without a systematic treatment. 

The use of the approach described in the previous paragraph enables us to deduce facts about $h$-dichotomies directly from those about exponential dichotomies. 

The paper is organized as follows. In Section~\ref{s-p}, we recall the basic notions that will appear throughout the paper. In Section~\ref{s-r}, we explore the connection between exponential and $h$-dichotomies. Afterwards, in Section~\ref{s-nn} we introduce the concepts of uniform $h$-noncriticality and $h$-expansiveness for an arbitrary invertible evolution family $T(t,s)$ and note that these can also be characterized in terms of uniform noncriticality and exponential expansiveness of $T_h(t,s)$. Finally,  in Section~\ref{s-mr} we obtain the main results of this paper. Firstly, we obtain the version of~\cite[Theorem 1]{Palmer} for invertible evolution families on Banach spaces (see Theorem~\ref{MT-1}) providing characterizations of exponential dichotomies. Using the other results of our paper, we then obtain the version of Theorem~\ref{MT-1} for $h$-dichotomies as a simple corollary (see Theorem~\ref{MT2}).

\section{Preliminaries}\label{s-p}

Throughout this note $X=(X, \|\cdot \|)$ will be a Banach space. By $\mathcal B(X)$ we denote the space of all bounded linear operators on $X$ equipped with the operator norm, which we also denote by $\| \cdot \|$.

We begin by recalling the notion of an evolution family.
\begin{definition}
Let $a_0\in \mathbb R\cup \{-\infty\}$. We say that $\mathcal T=\{T(t, s): \ t\ge s> a_0\}\subset \mathcal B(X)$ is an \emph{evolution family} if the following holds:
\begin{itemize}
\item $T(t, t)=\Id$ for $t> a_0$, where $\Id$ denotes the identity operator on $X$;
\item for $t\ge s\ge r> a_0$,
\[
T(t, s)T(s, r)=T(t, r);
\]
\item for $s>a_0$ and $v\in X$, $t\mapsto T(t, s)v$ is continuous on $[s, \infty)$.
 \end{itemize}
In addition, if $T(t, s)$ is an invertible operator for each $t\ge s> a_0$, we say that $\mathcal T$ is an \emph{invertible} evolution family.
\end{definition}

\begin{remark}\label{rem1}
Let $\mathcal T=\{T(t, s): \ t\ge s> a_0\}\subset \mathcal B(X)$ be an invertible evolution family. We can define $T(t, s)$ for $a_0< t<s$ by
\[
T(t, s):=T(s, t)^{-1}.
\]
Thus, in this case, we can view $\mathcal T$ as a family $\{T(t, s): \ t, s> a_0\}$.
\end{remark}

\begin{definition}
A \emph{growth rate} is any bijective increasing map $h\colon (a_0, \infty)\to (0, \infty)$, where $a_0\in \mathbb R\cup \{-\infty\}$.
\end{definition}

\begin{definition}
Let $h\colon (a_0, \infty)\to (0, \infty)$ be a growth rate and $\mathcal T=\{T(t, s): \ t\ge s> a_0\}\subset \mathcal B(X)$ an evolution family. We say that $\mathcal T$ exhibits \emph{$h$-bounded growth} on an interval $J\subset (a_0, \infty)$ if there exist $K, \mu>0$ such that 
\begin{equation}\label{hbg}
\|T(t, s)\| \le K\left (\frac{h(t)}{h(s)}\right )^\mu, \quad \text{for $t, s\in J$ with $t\ge s$.}
\end{equation}
\end{definition}

\begin{definition}
Let $h\colon (a_0, \infty)\to (0, \infty)$ be a growth rate and $\mathcal T=\{T(t, s): \ t\ge s> a_0\}\subset \mathcal B(X)$ an invertible evolution family. We say that $\mathcal T$ exhibits \emph{$h$-bounded decay} on an interval $J\subset (a_0, \infty)$ if there exist $K, \mu>0$ such that 
\begin{equation}\label{hbd}
\|T(t, s)\| \le K\left (\frac{h(s)}{h(t)}\right )^\mu, \quad \text{for $t, s\in J$ with $t\le s$.}
\end{equation}
\end{definition}

\begin{remark}
In the particular case where $h(t)=e^t$, the notion of an $h$-bounded growth coincides with the classical notion of a bounded growth (see~\cite[Definition 2.1.(iv)]{SS}). We note that this requirement is equivalent to that in~\cite[p.334]{MRS} which requires that there are $K\ge 0$ and $\mu \in \mathbb R$ such that~\eqref{hbg} holds, as for $\mu\le 0$ we have $e^{\mu (t-s)}\le 1$ for $t\ge s$.

In addition, in the case where $h(t)=e^t$, $h$-bounded decay will be called bounded decay.
\end{remark}

\begin{remark}
\begin{enumerate}
\item Assume that an evolution family $\mathcal T=\{T(t, s): \ t\ge s> a_0\}\subset \mathcal B(X)$ is generated by a nonautonomous linear equation
\begin{equation}\label{LDE}
    x'=A(t)x \quad t>a_0,
\end{equation}
where $A\colon (a_0, \infty)\to \mathcal B(X)$ is  a continuous map. In this case, \eqref{hbg} implies that
\begin{equation}\label{1546}
\|x(t)\|\le K\left (\frac{h(t)}{h(s)}\right )^\mu \|x(s)\|, \quad \text{for $t, s\in J$ with $t\ge s$,}
\end{equation}
where $t\mapsto x(t)$ is any solution of~\eqref{LDE}. This follows immediately, taking into account that $T(t, s)x(s)=x(t)$. Conversely, take any $s\in J$, $v\in X$ and let $x\colon (a_0, \infty)\to X$ be the solution of~\eqref{LDE} with $x(s)=v$.
From~\eqref{1546} we obtain
\[
\|T(t, s)v\|\le K\left (\frac{h(t)}{h(s)}\right )^\mu \|v\|, \quad \text{for $t\in J$ with $t\ge s$.}
\]
Since $v$ and $s$ were arbitrary, we conclude that~\eqref{hbg} holds.
The same discussion applies in relation to~\eqref{hbd}.
\item We note that an evolution family  $\mathcal T=\{T(t, s): \ t\ge s> a_0\}\subset \mathcal B(X)$ exhibits bounded growth if and only if there exist $C, T>0$ such that 
\begin{equation}\label{hbg-alt}
\|T(t, s)\| \le C, \quad \text{for $t, s\in J$ with $t\in [s, s+T]$.}
\end{equation}
Indeed, it follows from~\eqref{hbg} that~\eqref{hbg-alt} holds with $T=1$ and $C=Ke^\mu$. Conversely, suppose that~\eqref{hbg-alt} holds with $C>1$ (we can increase $C$ to achieve this) and take arbitrary $t, s\in J$ with $t\ge s$. Let $n\in \mathbb N_0$ be such that 
$t-s=nT+r$ with $0\le r<T$. Then
\[
\begin{split}
\|T(t, s)\|&=\|T(s+nT+r, s)\|\\
&=\|T(s+nT+r, s+nT)T(s+nT, s)\|\\
&\le C\|T(s+nT, s)\| \\
&\le C^{n+1}\\
&=e^{\frac{t-s+T-r}{T}\ln C}\\
&\le Ce^{(t-s)\frac{\ln C}{T}},
\end{split}
\]
which implies that~\eqref{hbg} holds with $K=C$ and $\mu=\frac{\ln C}{T}>0$.
\end{enumerate}
\end{remark}

We introduce the notion of a $h$-dichotomy for evolution families.
\begin{definition}\label{dich}
Let $h\colon (a_0, \infty)\to (0, \infty)$ be a growth rate and $\mathcal T=\{T(t, s): \ t\ge s>a_0\}\subset \mathcal B(X)$ an evolution family. We say that $\mathcal T$ admits an \emph{$h$-dichotomy} on an interval $J\subset (a_0, \infty)$ if there is a family $\{P(t): \ t\in J\}$ of bounded projections on $X$ and constants $D, \lambda >0$ such that:
\begin{enumerate}
\item for $t, s\in J$ with $t\ge s$,
\begin{equation}\label{invariance}
P(t)T(t, s)=T(t, s)P(s),
\end{equation}
and \begin{equation}\label{inv}T(t, s)\rvert_{\Ker P(s)}\colon \Ker P(s)\to \Ker P(t) \quad  \text{is invertible};\end{equation}
\item for $t, s\in J$ with $t\ge s$,
\begin{equation}\label{hd1}
\|T(t, s)P(s)\| \le D\left (\frac{h(t)}{h(s)}\right )^{-\lambda};
\end{equation}
\item for $t, s\in J$ with $t\le s$,
\begin{equation}\label{hd2}
\|T(t, s)(\Id-P(s))\| \le D\left (\frac{h(s)}{h(t)}\right)^{-\lambda},
\end{equation}
where 
\[
T(t, s):=\left (T(s, t)\rvert_{\Ker P(t)}\right )^{-1}\colon \Ker P(s)\to \Ker P(t).
\]
\end{enumerate}
\end{definition}
\begin{remark}
In the particular case where $h(t)=e^t$, the notion of a $h$-dichotomy coincides with the notion of an exponential dichotomy as introduced by Henry (see~\cite[Definition 7.6.1]{Henry}).
\end{remark}

\section{Time-rescaling}\label{s-r}
Throughout this section, we take a growth rate $h\colon (a_0, \infty)\to (0, \infty)$ and an evolution family
$\mathcal T=\{T(t, s): \ t\ge s>a_0\} \subset \mathcal B(X)$.  Set
\[
T_h(t, s):=T(h^{-1}(e^t), h^{-1}(e^s)), \quad t, s\in \mathbb R, \ t\ge s.
\]
\begin{remark}
Observe that $e^s>0$ for any $s\in \mathbb R$. Consequently, for $s\in \mathbb R$, $e^s$ belongs to the domain of $h^{-1}\colon (0, \infty)\to (a_0, \infty)$ and $h^{-1}(e^s)>a_0$. Therefore, since $h^{-1}$ is increasing we have $h^{-1}(e^t)\ge h^{-1}(e^s)>a_0$ for any $t, s\in \mathbb R$ with $t\ge s$. We conclude that $T_h(t, s)$ is well-defined.
\end{remark}
We then have the following simple observation.
\begin{proposition}
$\mathcal T^h:=\{T_h(t, s): \ t\ge s>-\infty\}$ is an evolution family.
\end{proposition}

\begin{proof}
For $t\in \mathbb R$, we have 
\[
T_h(t, t)=T(h^{-1}(e^t), h^{-1}(e^t))=\Id.
\]
Moreover, for $t\ge s\ge r$ we have
\[
\begin{split}
T_h(t, r)=T(h^{-1}(e^t), h^{-1}(e^r)) &=T(h^{-1}(e^t), h^{-1}(e^s))T(h^{-1}(e^s), h^{-1}(e^r))\\
&=T_h(t, s)T_h(s, r).
\end{split}
\] Here, we used $h^{-1}(e^t)\ge h^{-1}(e^s)\ge h^{-1}(e^r)$ (as $h^{-1}$ is increasing).

Finally, for all $s\in \mathbb R$ and $v\in X$, the map $t\mapsto T_h(t, s)v$ is continuous on $[s, \infty)$ as it is a composition of continuous maps $[h^{-1}(e^s), \infty)\ni t\mapsto T(t, h^{-1}(e^s))v$ and $t\mapsto h^{-1}(e^t)$.
\end{proof}
\begin{remark}
We notice that $\mathcal T^h$ is invertible provided that $\mathcal T$ is invertible. Indeed, for $t, s\in \mathbb R$ with $t\ge s$, the inverse of $T_h(t, s)$ is $T(h^{-1}(e^s), h^{-1}(e^t))$ (see Remark~\ref{rem1}).
\end{remark}

\begin{remark}
An analogous construction to that of $\mathcal T^h$ has been performed in the case of discrete time by Dragi\v cevi\' c and Silva~\cite[Eq.(5)]{DS} (building on the work~\cite{DSS} for polynomial dichotomies), and the same idea is briefly outlined in the work by Pe$\tilde{\text{n}}$a and  Rivera Villagran~\cite[p.3]{PV}.
\end{remark}
The main motivation for introducing $\mathcal T^h$ is the following result.
\begin{proposition}\label{P2}
Let $J=[a_0^*, \infty)$ for $a_0^*>a_0$.
The following holds:
\begin{enumerate}
\item[(a)] $\mathcal T$ exhibits $h$-bounded growth on $J$ if and only if $\mathcal T^h$ exhibits a bounded growth on $[\ln h(a_0^*), \infty)$;
\item[(b)] $\mathcal T$ admits a $h$-dichotomy on $J$ if and only if $\mathcal T^h$ admits an exponential dichotomy on $[\ln h(a_0^*), \infty)$.
\end{enumerate}
\end{proposition}
\begin{proof}
$(a)$ Suppose that $\mathcal T$ exhibits $h$-bounded growth on $J$ and let $K, \mu>0$ be such that~\eqref{hbg} holds. Then
\[
\|T_h(t, s)\|=\|T(h^{-1}(e^t), h^{-1}(e^s))\|\le K\left (\frac{h(h^{-1}(e^t))}{h(h^{-1}(e^s))}\right )^\mu=K e^{\mu (t-s)}
\]
for $t\ge s\ge \ln h(a_0^*)$, which implies that $h^{-1}(e^t)\ge h^{-1}(e^s)\ge a_0^*$. Hence, $\mathcal T^h$ exhibits a bounded growth on $[\ln h(a_0^*), \infty)$. The converse implication can be obtained in an analogous manner, noting that 
\begin{equation}\label{identity}
T(t, s)=T_h(\ln h(t), \ln h(s)), \quad t\ge s>a_0.
\end{equation}

$(b)$ Suppose that $\mathcal T$ admits a $h$-dichotomy on $J$ and let $P(t)$, $t\in J$ and $D, \lambda >0$ be as in Definition~\ref{dich}. Set
\[
\tilde P(t):=P(h^{-1}(e^t)), \quad t\ge \ln h(a_0^*).
\]
We first observe that
\[
\begin{split}
\tilde P(t)T_h(t, s) &=P(h^{-1}(e^t))T(h^{-1}(e^t), h^{-1}(e^s)) \\
&=T(h^{-1}(e^t), h^{-1}(e^s))P(h^{-1}(e^s))\\
&=T_h(t, s)\tilde P(s),
\end{split}
\]
for $t\ge s\ge \ln h (a_0^*)$. In addition, \[T_h(t, s)\rvert_{\Ker \tilde P(s)}=T(h^{-1}(e^t), h^{-1}(e^s))\rvert_{\Ker P(h^{-1}(e^s))}\colon \Ker \tilde P(s)\to \Ker \tilde P(t)\]
is invertible.

Secondly, it follows from~\eqref{hd1} that 
\[
\begin{split}
\|T_h(t, s)\tilde P(s)\| &=\|T(h^{-1}(e^t), h^{-1}(e^s))P(h^{-1}(e^s))\|\\
&\le D\left (\frac{h(h^{-1}(e^t))}{h(h^{-1}(e^s))}\right )^{-\lambda}\\
&=De^{-\lambda (t-s)},
\end{split}
\]
for $t\ge s\ge \ln h(a_0^*)$. Similarly, \eqref{hd2} gives
\[
\|T_h(t, s)(\Id-\tilde P(s))\|\le De^{-\lambda (s-t)},
\]
for $\ln h(a_0^*)\le t\le s$. We conclude that $\mathcal T^h$ admits an exponential dichotomy on $[\ln h(a_0^*), \infty)$. 
The converse implication can be obtained similarly by relying on~\eqref{identity}. 
\end{proof}
\begin{remark}\label{decay}
Similarly to the proof of Proposition~\ref{P2}(a), one can show that for an invertible evolution family $\mathcal T$ the following holds: $\mathcal T$ exhibits $h$-bounded decay on $J$ if and only if $\mathcal T^h$ exhibits bounded decay on $[\ln h(a_0^*), \infty)$.
\end{remark}

\section{Noncritical uniformity and expansiveness}\label{s-nn}
We introduce the concept of $h$-expansivity for invertible evolution families. 
\begin{definition}
Let $\mathcal T=\{T(t, s): \ t, s> a_0\}\subset \mathcal B(X)$ be an invertible evolution family and $h\colon (a_0, \infty)\to (0, \infty)$ be a growth rate. We say that $\mathcal T$ is \emph{$h$-expansive} on an interval $J\subset (a_0, \infty)$ if there exist $L, \beta>0$ such that 
\begin{equation}\label{hexp}
\|v\|\le L\left ( \left (\frac{h(t)}{h(a)}\right )^{-\beta}\|T(a, t)v\|+\left (\frac{h(b)}{h(t)}\right )^{-\beta}\|T(b, t)v\|\right ),
\end{equation}
for $v\in X$ and $a\le t\le b$ with $[a, b]\subset J$.
\end{definition}
\begin{remark}
In the case $h(t)=e^t$, we will say that $\mathcal T$ is \emph{exponentially expansive}. For evolution families arising from nonautonomous ordinary differential equations, this definition coincides with~\cite[Definition 5]{Palmer}.
\end{remark}
\begin{proposition}\label{P3}
Let $\mathcal T=\{T(t, s): \ t\ge s>a_0\}\subset \mathcal B(X)$ be an invertible evolution family and $h\colon (a_0, \infty)\to (0, \infty)$ be a growth rate. Then $\mathcal T$ is $h$-expansive on $[a_0^*, \infty)$, $a_0^*>a_0$ if and only if $\mathcal T^h$ is exponentially expansive on $[\ln h(a_0^*), \infty)$.
\end{proposition}

\begin{proof}
Suppose that $\mathcal T$ is $h$-expansive on $[a_0^*, \infty)$ and take an arbitrary $[c, d]\subset [\ln h(a_0^*), \infty)$. Let $a_0^*\le a \le b$ be such that $\ln h(a)=c$ and $\ln h(b)=d$. It follows from~\eqref{hexp} that for any $t\in [c, d]$ and $v\in X$ we have
\[
\begin{split}
\|v\| & \le L\bigg ( \left (\frac{h(h^{-1}(e^t))}{h(h^{-1}(e^c))}\right )^{-\beta}\|T(h^{-1}(e^c), h^{-1}(e^t))v\|\\
&\phantom{\le}+\left (\frac{h(h^{-1}(e^d))}{h(h^{-1}(e^t))}\right )^{-\beta}\|T(h^{-1}(e^d), h^{-1}(e^t))v\|\bigg )
\end{split},
\]
as $h^{-1}(e^t)\in [a, b]=[h^{-1}(e^c), h^{-1}(e^d)]\subset [a_0^*, \infty)$. Hence,
\[
\|v\| \le L(e^{-\beta(t-c)}\|T_h(c, t)v\|+e^{-\beta(d-t)}\|T_h(d, t)v\|).
\]
We conclude that $\mathcal T^h$ is exponentially expansive on $[\ln h(a_0^*), \infty)$. The converse implication can be established in an analogous manner on the basis of~\eqref{identity}.

\end{proof}
We now introduce the concept of uniform $h$-noncriticality for evolution families. 
\begin{definition}
Let $\mathcal T=\{T(t, s): \ t\ge s> a_0\}\subset \mathcal B(X)$ be an invertible evolution family and $h\colon (a_0, \infty)\to (0, \infty)$ be a growth rate. We say that $\mathcal T$ is \emph{uniformly $h$-noncritical} on $[a_0^*, \infty)$ for $a_0^*>a_0$ if there exist $\theta \in (0, 1)$ and $C>0$ such that 
\begin{equation}\label{huc}
\|v\|\le \theta \sup\{\|T(u, t)v\|:  |\ln h(u)-\ln h(t)|\le C\},
\end{equation}
for all $v\in X$ and  $t$ such that $h(t)\ge e^{C} h(a_0^*)$.
\end{definition}
\begin{remark}
In the case $h(t)=e^t$, we will say that $\mathcal T$ is \emph{uniformly noncritical}. For evolution families arising from nonautonomous ordinary differential equations, this definition coincides with~\cite[Definition 4]{Palmer}. This notion was introduced for nonlinear systems by  Krasovski~\cite{K} and later adapted to linear systems by  Massera and  Sch\"{a}ffer~\cite{MS}. 
\end{remark}

The following result follows easily from the previous definition.
\begin{proposition}\label{P4}
Let $\mathcal T=\{T(t, s): \ t\ge s>a_0\}\subset \mathcal B(X)$ be an invertible evolution family and $h\colon (a_0, \infty)\to (0, \infty)$ be a growth rate. Then $\mathcal T$ is uniformly $h$-noncritical on $[a_0^*, \infty)$, $a_0^*>a_0$ if and only if $\mathcal T^h$ is uniformly noncritical on $[\ln h(a_0^*), \infty)$.
\end{proposition}
\begin{proof}
Assume that $\mathcal T$ is uniformly $h$-noncritical and let $\theta \in (0, 1)$ and $C>0$ be as in the previous definition.
Note that 
\[
\begin{split}
&\sup\{\|T_h(u, t)v\|: \ |u-t|\le C\}\\
&=\sup\{\|T(h^{-1}(e^u), h^{-1}(e^t))v\|: \ |\ln h(h^{-1}(e^u))-\ln h(h^{-1}(e^t))|\le C\}.
\end{split}
\]
Consequently, for $t\ge C+\ln h(a_0^*)$ (so that $h(h^{-1}(e^t))\ge e^Ch(a_0^*))$ and $v\in X$ we have 
\[
\begin{split}
\|v\| &\le \theta \sup \{\|T(s,  h^{-1}(e^t))v\|: \ |\ln h(s)-\ln h(h^{-1}(e^t))|\le C\}\\
&=\theta \sup\{\|T(h^{-1}(e^u), h^{-1}(e^t))\|: \ |\ln h(h^{-1}(e^u))-\ln h(h^{-1}(e^t))|\le C\}\\
&=\theta \sup\{\|T_h(u, t)v\|: \ |u-t|\le C\},
\end{split}
\]
where in the second step, we made the change of variables $u=\ln h(s)$.
We conclude that $\mathcal T^h$ is uniformly noncritical on $[\ln h(a_0^*), \infty)$. The converse can be established similarly. 
\end{proof}
\section{Main results}\label{s-mr}
\subsection{Characterization of exponential dichotomies}
The following is the first main result of our paper. Its proof is inspired by the proof of~\cite[Proposition 1, p.14]{Coppel}.
\begin{theorem}\label{MT-1}
 Let $\mathcal T=\{T(t, s): \ t\ge s> a_0\}\subset \mathcal B(X)$ be an invertible evolution family and $a_0^*>a_0$ such that the following holds:
\begin{itemize}
\item $\mathcal T$ exhibits bounded growth and decay on $[a_0^*, \infty)$;
\item there exists a finite-dimensional subspace $Z\subset X$ such that 
\begin{equation}\label{split}
X=\mathcal S\oplus Z,
\end{equation}
where 
\begin{equation}\label{mathS}
\mathcal S:=\left \{v\in X: \ \sup_{t\ge a_0^*}\|T(t, a_0^*)v\|<+\infty \right \}. 
\end{equation}
\end{itemize}

Then the following statements are equivalent:
\begin{enumerate}
\item[(a)] $\mathcal T$ admits an exponential dichotomy on $[a_0^*, \infty)$;
\item[(b)] $\mathcal T$ is exponentially expansive on $[a_0^*, \infty)$;
\item[(c)] $\mathcal T$ is uniformly noncritical on $[a_0^*, \infty)$.
\end{enumerate}

\end{theorem}

\begin{proof}
$(a)\implies (b)$ Let $P(t)$, $t\ge a_0^*$ and $D, \lambda>0$ be as in Definition~\ref{dich} (for $h(t)=e^t$). Then, for every $v\in X$ and $[a, b]\subset [a_0^*, \infty)$ we have
\[
v=P(t)v+(\Id-P(t))v=T(t, a)P(a)T(a, t)v+T(t, b)(\Id- P(b))T(b, t)v,
\]
for $t\in [a, b]$. Consequently, by~\eqref{hd1} and~\eqref{hd2} we have
\[
\begin{split}
\|v\| &\le \|T(t, a)P(a)T(a, t)v\|+\|T(t, b)(\Id- P(b))T(b, t)v\|\\
&\le De^{-\lambda (t-a)}\|T(a, t)v\|+De^{-\lambda (b-t)}\|T(b, t)v\|,
\end{split}
\]
for $t\in [a, b]$. Hence, \eqref{hexp} holds with $\beta=\lambda$, $L=D$ and $h(t)=e^t$. We conclude that $\mathcal T$ is exponentially expansive on $[a_0^*, \infty)$.

$(b)\implies (c)$ Suppose that $\mathcal T$ is exponentially expansive on $[a_0^*, \infty)$ and let $L, \beta>0$ be such that~\eqref{hexp} holds with $h(t)=e^t$.  Choose $C>0$ sufficiently large so that $\theta:=2Le^{-\beta C}<1$. For any $v\in X$ and $t$ such that $t\ge a_0^*+C$, by~\eqref{hexp} we have
\[
\begin{split}
\|v\|&\le L(e^{-\beta C}\|T(t-C, t)v\|+e^{-\beta C}\|T(t+C, t)v\|)\\
&\le \theta \sup \{\|T(u, t)v\|: \ |u-t|\le C\},
\end{split}
\]
as $t\in [t-C, t+C]\subset [a_0^*, \infty)$. We conclude that $\mathcal T$ is uniformly noncritical on $[a_0^*, \infty)$.

$(c)\implies (a)$  Suppose that $\mathcal T$ is uniformly noncritical on $[a_0^*, \infty)$, and let $C>0$ and $\theta \in (0, 1)$ be such that 
\begin{equation}\label{912}
 \|v\|\le \theta \sup \{\|T(u, t)v\|: \ |u-t|\le C\},  \quad \text{for $v\in X$ and $t\ge a_0^*+C$.} 
\end{equation}
Moreover, let $K, \mu>0$ be such that~\eqref{hbg} and~\eqref{hbd} hold with $h(t)=e^t$.

 Let
\[
\mathcal S(s):=\left \{v\in X: \ \sup_{t\ge s}\|T(t, s)v\|<+\infty \right \}, \quad s\ge a_0^*.
\]
Since $T(s, a_0^*)\mathcal S=\mathcal S(s)$, by~\eqref{split} we have 
\begin{equation}\label{split2}
X=\mathcal S(s)\oplus Z(s), \quad \text{for $s\ge a_0^*$, where $Z(s):=T(s, a_0^*)Z$.}
\end{equation}
Here $T(s, a_0^*)Z$ denotes the image of the subspace $Z$ under the action of $T(s, a_0^*)$. Note that $\mathcal S(a_0^*)=\mathcal S$ and $Z(a_0^*)=Z$.
Fix $s\ge a_0^*$ and $v\in  \mathcal S(s)\setminus \{0\}$. Then
\[
0<\varrho:=\sup_{t\ge s}\|T(t, s)v\|<+\infty. 
\]
For $t\ge s+C$, it follows from~\eqref{912} that 
\[
\begin{split}
\|T(t, s)v\| &\le \theta \sup \{\|T(u, t)T(t, s)v\|: \ |u-t|\le C\} \\
&\le \theta \sup \{\|T(u, s)v\|: \ u\ge s\}=\theta \varrho,
\end{split}
\] 
yielding
\[
\sup_{t\ge s+C}\|T(t, s)v\| \le \theta \varrho < \varrho,
\]
since $\theta \in (0, 1)$.
Therefore, \[ \varrho=\sup_{t\in [s, s+C]}\|T(t, s)v\|, \] and, consequently, $\varrho \le Ke^{\mu C}\|v\|$ (where we used~\eqref{hbg} with $h(t)=e^t$).  We conclude that 
\begin{equation}\label{924}
\|T(t, s)v\| \le D\|v\| \quad \text{for $t\ge s$ and $v\in \mathcal S(s)$,}
\end{equation}
where $D\ge 1$ is independent of $s$ and $v$.  As in the proof of~\cite[Proposition 1, p.15]{Coppel}, \eqref{924} (together with~\eqref{912}) implies that
\begin{equation}\label{ed1}
\|T(t, s)v\|\le Be^{-\alpha (t-s)}\|v\| \quad \text{for $t\ge s\ge a_0^*$ and $v\in \mathcal S(s)$,}
\end{equation}
where $B:=\theta^{-1}D$ and $\alpha:=-C^{-1}\ln \theta$. Indeed, take $v\in \mathcal S(s)$,  $t\ge s\ge a_0^*$ and let $n\in \mathbb N_0$ be such that $s+nC\le t<s+(n+1)C$. By~\eqref{912}, 
\begin{equation}\label{aux1}
\begin{split}
\|T(t, s)v\| &\le \theta \sup \{\|T(u, t)T(t, s)v\|: \ |u-t|\le C\}\\
&=\theta \sup \{\|T(u, s)v\|: \ |u-t| \le C\}.
\end{split}
\end{equation}
For $u \in [t-C, t+C]$, we have (using~\eqref{912} again) that 
\begin{equation}\label{aux2}
\|T(u, s)v\| \le \theta \sup \{ \|T(r, s)v\|: \ |r-u| \le C\}
\end{equation}
Noting that $r\in [t-2C, t+2C]$ for each $r$ with $|r-u| \le C$, we obtain from~\eqref{aux1} and~\eqref{aux2} that 
\[
\|T(t, s)v\| \le \theta^2 \sup \{\|T(u, s)v\|: \ |u-t| \le 2C\}.
\]
Iterating,
\[
\|T(t, s)v\| \le \theta^n \sup \{\|T(u, s)v\|: \ |u-t| \le nC\}\le D\theta^n \|v\|, 
\]
where in the last inequality, we used~\eqref{924}. This easily implies~\eqref{ed1}.

For $v\in Z$ with $\|v\|=1$, the map $t\mapsto \|T(t, a_0^*)v\|$ is unbounded (as otherwise we would have $v\in\mathcal S\cap Z =\{0\}$) and, consequently, there is a least value $t_1=t_1(v)>a_0^*$ such that $\|T(t_1, a_0^*)v\|=\theta^{-1}D$. We claim that $v\mapsto t_1(v)$ is bounded. Otherwise, there is a sequence $(v_n)_n\subset Z$ with $\|v_n\|=1$ such that $t_1(v_n)\to \infty$. By the  compactness of the unit sphere in $Z$ (here we use that $Z$ is finite-dimensional), we can assume without loss of generality that $v_n\to v$, where $v\in Z$ and $\|v\|=1$.  Since $T(t, a_0^*)\in \mathcal B(X)$, we have 
\[
T(t, a_0^*)v_n\to T(t, a_0^*)v \quad \text{for $t\ge a_0^*$.}
\]
As 
\[
\|T(t, a_0^*)v_n\| <\theta^{-1}D \quad \text{for $n\in \mathbb N$ and $a_0^*\le t<t_1(v_n)$,}
\]
we have
\[
\|T(t, a_0^*)v\| = \lim_{n\to \infty}\|T(t, a_0^*)v_n\| \le \theta^{-1}D, \quad t\ge a_0^*.\]
This implies that the map $t\mapsto \|T(t, a_0^*)v\|$ is bounded, which yields a contradiction. This proves our claim. Let $t_1:=\sup_{v\in Z, \|v\|=1}t_1(v) \in (a_0^*, \infty)$.

As in~\cite[p.15]{Coppel} we have 
\[
\|T(t, a_0^*)v\|\le Be^{-\alpha (s-t)}\|T(s, a_0^*)v\|, \quad \text{for $v\in Z$ and $t_1\le t\le s$.}
\]
Hence,
\begin{equation}\label{ed2}
    \|T(t,s)v\|\le Be^{-\alpha (s-t)}\|v\|, \quad \text{for $v\in Z(s)$ and $t_1\le t\le s$.}
\end{equation}
On the other hand, for $a_0^*\le t\le s\le t_1$ we have from~\eqref{hbd} that
\begin{equation}\label{1517}
\|T(t, s)v\|\le Ke^{\mu (s-t)}\|v\| \le Ke^{(\mu+\alpha)(t_1-  a_0^*)}e^{-\alpha (s-t)}\|v\|,
\end{equation}
for $v\in Z(s)$ as
\[
e^{\mu (s-t)}=e^{(\mu+\alpha) (s-t)}e^{-\alpha(s-t)}\le e^{(\mu+\alpha)(t_1-  a_0^*)}e^{-\alpha (s-t)}.
\]
Finally, for $a_0^*\le t\le t_1\le s$ using~\eqref{hbd} and~\eqref{ed2} we have
\begin{equation}\label{1518}
\begin{split}
    \|T(t, s)v\|=\|T(t, t_1)T(t_1, s)v\|&\le Ke^{\mu (t_1-t)}\|T(t_1, s)v\| \\
    &\le KBe^{\mu (t_1-t)}e^{-\alpha (s-t_1)}\|v\|\\
    &= KBe^{(\mu+\alpha)(t_1-t)}e^{-\alpha (s-t)}\|v\| \\
    &\le KBe^{(\mu+\alpha)(t_1-a_0^*)}e^{-\alpha (s-t)}\|v\|,
    \end{split}
\end{equation}
for $v\in Z(s)$. From~\eqref{ed2}, \eqref{1517} and~\eqref{1518} we conclude that there exists $\tilde B>0$ such that
\begin{equation}\label{1519}
     \|T(t,s)v\|\le \tilde Be^{-\alpha (s-t)}\|v\|, \quad \text{for $v\in Z(s)$ and $a_0^*\le t\le s$.}
\end{equation}
It follows from~\eqref{hbg}, \eqref{ed1} and~\eqref{1518} (together with~\cite[Lemma 4.2]{MRS}) that $\mathcal T$ admits an exponential dichotomy on $[a_0^*, \infty)$ with respect to projections $P(t)\colon X\to \mathcal S(t)$ associated to~\eqref{split2}.
\end{proof}

\begin{remark}\label{sevremarks}
\begin{enumerate}
\item When $X$ is finite-dimensional, the assumption that there is a finite-dimensional subspace $Z\subset X$ satisfying~\eqref{split} is automatically satisfied.
\item When the evolution family $T(t, s)$ is associated with an ordinary differential equation \[x'=A(t)x, \] where $A\colon (a_0, \infty)\to \mathcal B(X)$ is continuous, we can remove the assumption that $T(t, s)$ exhibits a bounded decay (see the argument in~\cite[p.13]{Coppel}).
\end{enumerate}
\end{remark}

The following example shows that the assumption that $\mathcal S$ has a finite-dimensional complement cannot be omitted in the statement of Theorem~\ref{MT-1} (even when $X$ is a Hilbert space).

\begin{example}\label{EX1}
Let $H=L^2(\mathbb R)$ be the Hilbert space consisting of all square-integrable functions $f\colon \mathbb R\to \mathbb C$ with respect to the Lebesgue measure on $\mathbb R$. For $t\in \mathbb R$, let $S(t)\colon H\to H$ be given by 
\[
(S(t)f)(x):=f(x-t), \quad x\in \mathbb R, \ f\in H.
\]
Then $S(t)$ is a linear isometry on $H$ for each $t\in \mathbb R$. Note that
\[
    S(0)=\Id \quad \text{and} \quad S(t+s)=S(t)S(s) \quad \text{for $t, s\in \mathbb R$.}
\]

Moreover, set $Y:=H\oplus H$  (which is also a Hilbert space).
For $t\in \mathbb R$, we define $T(t)\colon Y\to Y$  by 
\[
T(t)(f, g)=(e^{-t}S(t)f, e^t S(t)g), \quad (f, g)\in Y.
\]
Clearly, $T(t)$ is a bounded linear operator on $Y$.
Let $\mathcal T=\{T(t, s): \ t, s\in \mathbb R\}$ be an invertible evolution family on $Y$ defined by
\[
T(t, s):=T(t-s), \quad t, s\in \mathbb R.
\]
We claim that $\mathcal T$ admits an exponential dichotomy. To this end, we consider a projection $P\colon Y\to Y$ given by 
\[
P(f, g)=(f, 0), \quad (f, g)\in Y,
\]
and set $P(t):=P$ for $t\in \mathbb R$. Clearly, \eqref{invariance} holds. Moreover, 
\[
\begin{split}
\|T(t, s)P(s)(f, g)\|=\|T(t-s)(f, 0)\| &=e^{-(t-s)}\|S(t-s)f\|\\
&=e^{-(t-s)}\|f\|\\
&\le e^{-(t-s)}\|(f, g)\|,
\end{split}
\]
for $t\ge s$ and $(f, g)\in Y$. Similarly, 
\[
\|T(t, s)(\Id-P(s))(f, g)\|\le e^{-(s-t)}\|(f, g)\|,
\]
for $t\le s$ and $(f, g)\in Y$. We conclude that $\mathcal T$ admits an exponential dichotomy.

Next, we consider the operator $M\colon \mathcal D(M)\to L^2(\mathbb R)$ defined by 
\[
(Mf)(x)=e^{2x}f(x) \quad x\in \mathbb R,
\]
defined on the domain $\mathcal D(M)$ consisting of all $f\in L^2(\mathbb R)$ such that $Mf\in L^2(\mathbb R)$. Notice that $\mathcal D(M)$ is dense in $L^2(\mathbb R)$ as it contains all continuous functions with bounded support. We observe that $M$ is a closed operator. Indeed, if $(f_n)_n$ is a sequence in $\mathcal D(M)$ such that $f_n\to f$ and $Mf_n \to g$ in $L^2(\mathbb R)$, then we can find a subsequence $(f_{n_k})_k$ of $(f_n)_n$ such that $f_{n_k}\to f$ and $Mf_{n_k}\to g$ almost everywhere. This easily implies that $e^{2x}f(x)=g(x)$ for a.e. $x\in \mathbb  R$. Consequently, $f\in \mathcal D(M)$ and $Mf=g$.

We take $X$ to the graph of $M$ (see~\cite[p.91]{Conway}): 
\[
X=\{(f, Mf): \ f\in \mathcal D(M)\}.
\]
Since $M$ is closed, we have that $X$ is a closed subspace of $Y$. Take $t\in \mathbb R$ and $(f, Mf)\in X$ and define $h:=e^{-t}S(t)f$. Note that 
\[
(Mh)(x)=e^{2x}h(x)=e^{2x}e^{-t}(S(t)f)(x)=e^{2x}e^{-t}f(x-t),
\]
for $x\in \mathbb R$. Since $f\in \mathcal D(M)$, one can easily see that $Mh\in L^2(\mathbb R)$ and $h\in \mathcal D(M)$. Moreover, 
\[
e^t (S(t)Mf)(x)=e^t e^{2(x-t)}f(x-t)=e^{-t}e^{2x}f(x-t)=Mh(x),
\]
for $x\in \mathbb R$.
Hence,
\[
T(t)(f, Mf)=(e^{-t}S(t)f, e^t S(t)Mf)=(h, Mh).
\]
This implies that $X$ is $T(t, s)$-invariant for arbitrary $t, s\in \mathbb R$. Moreover, 
\begin{equation}\label{436n}
\begin{split}
\|T(t)(f, Mf)\|^2 &=\|( h, M  h)\|^2\\
&=\|h\|^2+\|M h\|^2 \\
&=e^{-2t}\|S(t)f\|^2+e^{2t}\|S(t)Mf\|^2\\
&=e^{-2t}\|f\|^2+e^{2t}\|Mf\|^2,
\end{split}
\end{equation}
for arbitrary $t\in \mathbb R$.
Note that for $(f, Mf)\in X\setminus \{0\}$ we have that $f\neq 0$ and $Mf\neq 0$. Consequently, \eqref{436n} implies that
\begin{equation}\label{458n}
\lim_{t\to \infty}\|T(t)(f, Mf)\|=\infty \quad \text{and} \quad \lim_{t\to -\infty}\|T(t)(f, Mf)\|=\infty,
\end{equation}
for each nonzero $(f, Mf)\in X$. This implies that an invertible evolution family $\bar{\mathcal T}=\{\bar T(t, s): \ t, s\in \mathbb R\}$ on $X$ defined by 
\[
\bar T(t, s):=T(t, s)\rvert_X \quad t, s\in \mathbb R
\]
does not admit an exponential dichotomy on an interval $[0, \infty)$. Indeed, assume that it does admit an exponential dichotomy on $[0, \infty)$ with respect to projections $P(t)$, $t\ge 0$. As
\[
\begin{split}
\Ima P(t) &=\left \{(f, Mf)\in X: \ \sup_{r\ge t}\|\bar T(r, t)(f, Mf)\|<+\infty \right \}\\
&=\left  \{(f, Mf)\in X: \ \sup_{r\ge t}\|T(r-t)(f, Mf)\|<+\infty \right \}\\
&=\left \{(f, Mf)\in X: \sup_{r\ge 0}\|T(r)(f, Mf)\|<+\infty \right\},
\end{split}
\]
we conclude from the first equality in~\eqref{458n} that $P(t)=0$ for each $t\ge 0$. Consequently, we have that there exist $D, \lambda>0$ such that 
\[
\|\bar T(t, s)(f, Mf)\|\ge \frac 1 D e^{\lambda (t-s)}\|(f, Mf)\|,
\]
for each $t\ge s\ge 0$ and $(f, Mf)\in X$. Hence, 
\[
\|T(-r)(f, Mf)\| \le De^{-\lambda r}\|(f, Mf)\|
\]
for each $r\ge 0$ and $(f, Mf)\in X$, which contradicts the second equality in~\eqref{458n}.

Since $\mathcal T$ admits an exponential dichotomy, it follows from the proof of Theorem~\ref{MT-1} that $\mathcal T$ is uniformly noncritical on $[0, \infty)$. Thus, there are $C>0$ and $\theta \in (0, 1)$ such that 
\[
\|(f, g)\| \le \theta \sup \{\|T(u, t)(f, g)\|: \ |u-t|\le C\}, \ \text{for $(f, g)\in Y$ and $t\ge C$.}
\]
In particular, 
\[
\|(f, Mf)\| \le \theta \sup \{\|T(u, t)(f, Mf)\|: \ |u-t|\le C\}, 
\]
for $(f, Mf)\in X$ and $t\ge C$, 
yielding 
\[
\|(f, Mf)\| \le \theta \sup \{\|\bar T(u, t)(f, Mf)\|: \ |u-t|\le C\},
\]
for $(f, Mf)\in X$ and $t\ge C$. Therefore, $\bar {\mathcal T}$ is uniformly noncritical on $[0, \infty)$. We conclude that $\bar{\mathcal T} $ is uniformly noncritical on $[0, \infty)$ although it does not admit an exponential dichotomy on $[0, \infty)$.

We finally observe that it follows from~\eqref{458n} that $\mathcal S$ in~\eqref{mathS} (with $a_0^*=0$) corresponding to $\bar{\mathcal T}$ consist only of a zero vector, which means that~\eqref{split} holds with $Z=X$. Therefore, we conclude that in the statement of Theorem~\ref{MT-1} we cannot relax the assumption that $Z$ is finite-dimensional by requiring only that $Z$ is closed.

\end{example}
\subsection{Characterization of $h$-dichotomies}
Now we obtain the version of Theorem~\ref{MT-1} for $h$-dichotomies. Its proof is a simple consequence of Theorem~\ref{MT-1} and the other results of this paper.
\begin{theorem}\label{MT2}
 Let $\mathcal T=\{T(t, s): \ t\ge s> a_0\}\subset \mathcal B(X)$ be an invertible evolution family, $h\colon (a_0, \infty)\to (0, \infty)$ a growth rate and $a_0^*>a_0$ such that the following holds:
\begin{itemize}
\item $\mathcal T$ exhibits $h$-bounded growth and decay on $[a_0^*, \infty)$;
\item there exists a finite-dimensional subspace $Z\subset X$ such that
\[
X=\mathcal S \oplus Z,
\]
where 
\[
\mathcal S:=\left \{v\in X: \ \sup_{t\ge a_0^*}\|T(t, a_0^*)v\|<+\infty \right \}. 
\]
\end{itemize}

Then the following statements are equivalent:
\begin{enumerate}
\item[(a)] $\mathcal T$ admits an $h$-dichotomy on $[a_0^*, \infty)$;
\item[(b)] $\mathcal T$ is $h$-expansive on $[a_0^*, \infty)$;
\item[(c)] $\mathcal T$ is uniformly $h$-noncritical on $[a_0^*, \infty)$.
\end{enumerate}

\end{theorem}

\begin{proof}
The desired conclusion follows directly from Theorem~\ref{MT-1} by noting the following:
\begin{itemize}
\item $\mathcal T^h$ exhibits bounded growth and decay on $[\ln h(a_0^*), \infty)$ (see Proposition~\ref{P2}(a) and Remark~\ref{decay});
\item $\mathcal T$ admits $h$-dichotomy on $[a_0^*, \infty)$ if and only if $\mathcal T^h$ admits exponential dichotomy on $[\ln h(a_0^*), \infty)$ (see Proposition~\ref{P2}(b));
\item $\mathcal T$ is $h$-expansive on $[a_0^*, \infty)$ if and only if $\mathcal T^h$ is exponentially expansive on $[\ln h(a_0^*), \infty)$ (see Proposition~\ref{P3});
\item $\mathcal T$ is uniformly $h$-noncritical on $[a_0^*, \infty)$ if and only if $\mathcal T^h$ is uniformly noncritical on $[\ln h(a_0^*), \infty)$ (see Proposition~\ref{P4});
\item 
\[
\mathcal S=\left \{v\in X: \ \sup_{t\ge \ln h(a_0^*)}\|T_h(t, \ln h(a_0^*))v\|<+\infty \right \}.
\]
\end{itemize}

\end{proof}

\section*{Acknowledgements}
It is a pleasure to thank all three referees for their careful reading of my paper and for their useful comments. I am also deeply grateful to Prof. M. Pituk for suggesting  Example~\ref{EX1}.

\vspace{7mm}

\noindent {\bf  Funding}  
This paper has been funded by European Union – NextGenerationEU-Statistical properties of random dynamical systems and other contributions to mathematical analysis and probability theory.

\vspace{3mm}

\noindent{\bf Data availability} This  manuscript has no associated data.

\section*{Declarations}

\noindent{\bf Conflict of interest}  The author declares that he has no known competing financial interests or personal
relationships that could have appeared to influence the work in this paper.

\end{document}